\documentclass[a4paper,12pt]{article}
\usepackage{times, url}
\usepackage{amsmath,amssymb,amsthm,amsfonts}

\newtheorem{thm}{Theorem}[section]
\newtheorem{lem}{Lemma}[section]
\newtheorem{rem}{Remark}[section]

\begin{document}
\setcounter{page}{1} 
\vspace{10mm}

\begin{center}
{\LARGE \bf  Some New Results On Even Almost Perfect Numbers Which Are Not Powers Of Two}
\vspace{8mm}

{\large \bf John Rafael M. Antalan$^1$ and Jose Arnaldo B. Dris$^2$}
\vspace{3mm}

$^1$ Department of Mathematics and Physics, Central Luzon State University \\ 
Science City of Munoz (3120) Nueva Ecija, Philippines \\
e-mail: \url{jrantalan@clsu.edu.ph}
\vspace{2mm}

$^2$ Department of Mathematics and Physics, Far Eastern University \\ 
Nicanor Reyes Street, Sampaloc, Manila, Philippines \\
e-mail: \url{josearnaldobdris@gmail.com}
\vspace{2mm}

\end{center}
\vspace{10mm}

\noindent
{\bf Abstract:} In this note, we present some new results on even almost perfect numbers which are not powers of two.  In particular, we show that $2^{r+1} < b$, if ${2^r}{b^2}$ is an even almost perfect number. 
\vspace{10mm}

\section{Introduction} 

Let $\sigma(x)$ denote the sum of divisors of $x$.  If $\sigma(y)=2y-1$, we say that $y$ is \emph{almost perfect}.

In \cite{Dris}, Dris gives the following criterion for almost perfect numbers in terms of the abundancy index $I(x) = \sigma(x)/x$:

\begin{thm}\label{Dris1}
Let $m$ be a positive integer. Then $m$ is almost perfect if and only if
$$\frac{2m}{m+1} \leq I(m) < \frac{2m+1}{m+1}.$$
\end{thm}

Dris also obtains the following result \cite{Dris}:

\begin{thm}\label{Dris2}
Let $M$ be a positive integer. Then $M$ is deficient if and only if
$$\frac{2M}{M+D(M)} \leq I(M) < \frac{2M+D(M)}{M+D(M)},$$
where $D(M)=2M-\sigma(M)$ is the deficiency of $M$.
\end{thm}

It is currently an open problem to determine if the only even almost perfect numbers are those of the form $2^k$, where $k \geq 1$.  (Note that $1$ is the single currently known odd almost perfect number, as $\sigma(1) = 2\cdot1 - 1 = 1$.)

Antalan and Tagle showed in \cite{AntalanTagle} that, if $M \neq 2^k$ is an even almost perfect number, then $M$ takes the form $M = {2^r}{b^2}$, where $b$ is an odd composite integer.  Antalan also proved in \cite{Antalan} that $3 \nmid M$.

\section{Main Results}

Our penultimate goal is, of course, to show that if $n$ is an even almost perfect number, then $n = 2^k$ for some positive integer $k$. \\

Assume to the contrary that there exists an even almost perfect number $M \neq 2^k$. By \cite{AntalanTagle}, $M$ then takes the form $M = {2^r}{b^2}$, where $r \geq 1$ and $b$ is an odd composite integer.  Note that $b^2$ is deficient, as it is a factor of the deficient number $M = {2^r}{b^2}$. \\

(The following proof for the assertion that $b^2$ is not almost perfect, is from \cite{MSE-Dris2}.) \\

Since $M$ is almost perfect, we have
$$(2^{r+1} - 1)\sigma(b^2) = \sigma(2^r)\sigma(b^2) = \sigma({2^r}{b^2}) = \sigma(M) = 2M - 1 = {2^{r+1}}b^2 - 1.$$
So we have
$$\sigma(b^2) = \frac{{2^{r+1}}b^2 - 1}{2^{r+1} - 1} = b^2 + \frac{b^2 - 1}{2^{r+1} - 1}.$$
Now,
$$2b^2 - \sigma(b^2) = b^2 -\frac{b^2 - 1}{2^{r+1} - 1}.$$
If $b^2$ is also almost perfect, then we have
$$1 = 2b^2 - \sigma(b^2) = b^2 -\frac{b^2 - 1}{2^{r+1} - 1},$$
which, since $b > 1$, gives
$$2^{r+1} - 1 = 1 \Longleftrightarrow r = 0.$$ \\

This contradicts $r \geq 1$.  Consequently, since $b^2$ is deficient, we can write $\sigma(b^2) = 2b^2 - c$, where $c > 1$. \\

Note that we have proved the following propositions: \\

\begin{lem}
Let $M = {2^r}{b^2}$ be an even almost perfect number, with $\sigma(b^2) = 2b^2 - c$.  Then
$$c = b^2 -\frac{b^2 - 1}{2^{r+1} - 1}.$$
\end{lem}

\begin{lem}
Let $M = {2^r}{b^2}$ be an even almost perfect number, with $\sigma(b^2) = 2b^2 - c$.  Then
$$c \geq \frac{2b^2 + 1}{3}.$$
\end{lem}

Notice that, since $b$ is an odd composite, and since $3 \nmid M$ (see \cite{Antalan}), then $b \geq 5\cdot{7} = 35$, so that we have the estimate $c \geq \frac{2\cdot{{35}^2} + 1}{3} = 817.\overline{333}$, which implies that $c \geq 819$ since $c$ is an odd integer. \\

Recall that the abundancy index of $x$ is defined to be the ratio $I(x) = \frac{\sigma(x)}{x}$.  We call a number $S$ solitary if the equation $I(S) = I(d)$ has exactly one solution $d = S$.  A sufficient (but not necessary) condition for $T$ to be solitary is $\gcd(T, \sigma(T)) = 1$, where $\gcd$ is the greatest common divisor function. \\

The following result was communicated to the second author by Dagal last October 4, 2015. \\

\begin{lem}
If ${2^r}{b^2}$ is an almost perfect number with $\gcd(2,b) = 1$ and $b > 1$, then $b^2$ is solitary.
\end{lem}

(Note: The proof that follows is different from that of Dagal's \cite{DagalYmas}.) \\

\begin{proof}
Since ${2^r}{b^2}$ is almost perfect, we have
$$(2^{r+1} - 1)\sigma(b^2) = \sigma(2^r)\sigma(b^2) = \sigma({2^r}{b^2}) = {2^{r+1}}{b^2} - 1.$$
We want to show that
$$\gcd(b^2, \sigma(b^2)) = 1.$$
It suffices to find a linear combination of $b^2$ and $\sigma(b^2)$ that is equal to $1$.  Such a linear combination is given by the equation
$$1 = (1 - 2^{r+1})\sigma(b^2) + {2^{r+1}}{b^2}.$$
\end{proof}

From the equation
$$1 = (1 - 2^{r+1})\sigma(b^2) + {2^{r+1}}{b^2}$$
we obtain
$${2^{r+1}}{\left(\sigma(b^2) - b^2\right)} = \sigma(b^2) - 1$$
so that
$$2^{r+1} = \frac{\sigma(b^2) - 1}{\sigma(b^2) - b^2} = 1 + \frac{b^2 - 1}{\sigma(b^2) - b^2}.$$ \\

This last equation gives the divisibility constraint in the following result:

\begin{lem}
If ${2^r}{b^2}$ is an almost perfect number with $\gcd(2,b) = 1$ and $b > 1$, then 
$$\left(\sigma(b^2) - b^2\right) \mid \left(b^2 - 1\right).$$
\end{lem}

Numbers $n$ such that $\sigma(n) - n$ divides $n - 1$ are listed in OEIS sequence A059046 \cite{OEIS-A059046}, the first $62$ terms of which are given below:
$$2, 3, 4, 5, 7, 8, 9, 11, 13, 16, 17, 19, 23, 25, 27, 29, 31, 32, 37, 41, 43, 47, 49, 53, 59, 61, 64, 67, 71, 73, 77, 79,$$
$$81, 83, 89, 97, 101, 103, 107, 109, 113, 121, 125, 127, 128, 131, 137, 139, 149, 151, 157, 163, 167, 169, 173,$$
$$179, 181, 191, 193, 197, 199, 211.$$ \\

\begin{rem}
Does OEIS sequence A059046 contain any odd squares $u^2$, with $\omega(u) \geq 2$? MSE user Charles (\url{http://math.stackexchange.com/users/1778}) checked and found that "`there are no such squares with $u^2 < {10}^{22}$."' \cite{MSE-Dris1}
\end{rem}

Suppose that $M = {2^r}{b^2}$ is an almost perfect number with $\gcd(2,b) = 1$ and $b > 1$.  Let us call $b^2$ the 
$\emph{odd part}$ of $M$. \\

The following result shows that distinct even almost perfect numbers (other than the powers of $2$) cannot share the same odd part. \\

\begin{lem}
Suppose that there exist at least two distinct even almost perfect numbers
$$M_1 = {2^{r_1}}{{b_1}^2}$$
and 
$$M_2 = {2^{r_2}}{{b_2}^2},$$
with $\gcd(2,b_1) = \gcd(2,b_2) = 1$, $b_1 > 1$, $b_2 > 1$, and $r_1 \neq r_2$.  Then $b_1 \neq b_2$.
\end{lem}

\begin{proof}
Assume to the contrary that $1 < b_1 = b_2 = b$.  This implies that $1 < {b_1}^2 = {b_2}^2 = b^2$, so that
$$\frac{{2^{r_1 + 1}}{b^2} - 1}{{2^{r_2 + 1}}{b^2} - 1}= \frac{2M_1 - 1}{2M_2 - 1} = \frac{\sigma(M_1)}{\sigma(M_2)} = \frac{(2^{r_1 + 1} - 1)\sigma(b^2)}{(2^{r_2 + 1} - 1)\sigma(b^2)} = \frac{2^{r_1 + 1} - 1}{2^{r_2 + 1} - 1}.$$ \\

Solving for $b^2$ gives
$$(2^{r_2 + 1} - 1)({2^{r_1 + 1}}{b^2} - 1) = (2^{r_1 + 1} - 1)({2^{r_2 + 1}}{b^2} - 1)$$
$${2^{r_1 + r_2 + 2}}{b^2} - {2^{r_1 + 1}}{b^2} - 2^{r_2 + 1} + 1 = {2^{r_1 + r_2 + 2}}{b^2} - {2^{r_2 + 1}}{b^2} - 2^{r_1 + 1} + 1$$
$$(2^{r_1 + 1} - 2^{r_2 + 1}){b^2} = {2^{r_1 + 1}}{b^2} - {2^{r_2 + 1}}{b^2} = 2^{r_1 + 1} - 2^{r_2 + 1}.$$ \\

By assumption, we have $r_1 \neq r_2$, so that $2^{r_1 + 1} - 2^{r_2 + 1} \neq 0$.  Finally, we get
$$b^2 = 1,$$
which is a contradiction.
\end{proof}

Since $b^2$ is composite, $\sigma(b^2) > b^2 + b + 1$.  In particular, we obtain
$$b^2 - b - 1 > 2b^2 - \sigma(b^2).$$

From the equation
$$2^{r+1} = 1 + \frac{b^2 - 1}{\sigma(b^2) - b^2}$$
and the inequality
$$b^2 + b + 1 < \sigma(b^2),$$
we obtain the following result: \\

\begin{thm}
If ${2^r}{b^2}$ is an almost perfect number with $\gcd(2,b) = 1$ and $b > 1$, then
$$r < \log_{2}{b} - 1.$$
\end{thm}

This last inequality implies that
$$2^r < 2^{r+1} < b < \sigma(b)$$
and
$$\sigma(2^r) = 2^{r+1} - 1 < b - 1 < b$$
so that we have
$$\frac{\sigma(2^r)}{b} < 1 < 2 < \frac{\sigma(b)}{2^r}.$$

Additionally, since $b^2$ is deficient, we can write $\sigma(b^2) = 2b^2 - c$, where we compute $c$ to be
$$c = b^2 - \frac{b^2 - 1}{\sigma(2^r)}$$
from which we obtain the upper bound
$$\frac{\sigma(b)}{b} < \frac{\sigma(b^2)}{b^2} < \frac{4}{3}.$$

(Note that $I(b^2) < 4/3$ implies $3 \nmid b$.  For suppose to the contrary that $I(b^2) < 4/3$ and $3 \mid b$.  Then $3^2 \mid b^2$, so that $13/9 = I(3^2) \leq I(b^2) < 4/3$, which is a contradiction. This approach provides an alternative to Antalan's proof \cite{Antalan}.)

Lastly, since $r \geq 1$ and $2 \mid 2^r$, then
$$\frac{3}{2} = \frac{\sigma(2)}{2} \leq \frac{\sigma(2^r)}{2^r},$$
so that we have the following series of inequalities: \\

\begin{thm}
If ${2^r}{b^2}$ is an almost perfect number with $\gcd(2,b) = 1$ and $b > 1$, then
$$\frac{\sigma(2^r)}{b} < 1 < \frac{\sigma(b)}{b} < \frac{4}{3} < \frac{3}{2} \leq \frac{\sigma(2^r)}{2^r} < 2 < \frac{\sigma(b)}{2^r}.$$
\end{thm}

We can obtain a tighter lower bound for $\sigma(b^2)/b^2$ via the following method (using the result from Dris \cite{Dris} cited earlier):

\begin{thm}
If ${2^r}{b^2}$ is an almost perfect number with $\gcd(2,b) = 1$ and $b > 1$, then
$$\frac{2b - 1}{2b - 2} < I(b^2).$$
In particular,
$$\sqrt{\frac{2b - 1}{2b - 2}} < I(b).$$
\end{thm}

\begin{proof}
We start with
$$\frac{2b^2 + 1}{3} \leq D(b^2) < b^2 - b - 1.$$
Since $D(b^2) \geq 819$, we can use the following bounds from \cite{Dris}:
$$\frac{2b^2}{b^2 + D(b^2)} < I(b^2) < \frac{2b^2 + D(b^2)}{b^2 + D(b^2)}.$$
This simplifies to
$$\frac{2b^2}{2b^2 - b - 1} < I(b^2) < \frac{9b^2 - 3b - 3}{5b^2 + 1},$$
from which it follows that
$$\frac{2b^2}{2b^2 - b - 1} = \frac{2b^2 - b - 1}{2b^2 - b - 1} + \frac{b + 1}{2b^2 - b - 1} = 1 + \frac{b + 1}{2b^2 - b - 1} = 1 + \frac{b + 1}{(2b + 1)(b - 1)},$$
of which the last quantity is bounded below by
$$1 + \frac{b + 1}{(2b + 1)(b - 1)} > 1 + \frac{b + 1}{2(b + 1)(b - 1)} = \frac{2b - 1}{2b - 2}.$$
The last assertion in the theorem follows from
$$\left(I(b)\right)^2 > I(b^2).$$
\end{proof}

Proceeding similarly as before, we can prove the following result.

\begin{thm}
If ${2^r}{b^2}$ is an almost perfect number with $\gcd(2,b) = 1$ and $b > 1$, then
\begin{itemize}
{
\item{$r = 1 \Longrightarrow 8/7 < I(b^2) < 4/3 \Longrightarrow 3 \nmid b$; and}
\item{$r > 1 \Longrightarrow I(b^2) < 8/7 \Longrightarrow 7 \nmid b$.}
}
\end{itemize}
\end{thm}

\begin{proof}
The details of the proof (as well as other relevant hyperlinks) are in the following MathOverflow post:
\url{http://mathoverflow.net/q/238824}.
\end{proof}

\section*{Acknowledgements} 

The authors would like to thank the anonymous referees for their valuable feedback and suggestions that helped improve the overall presentation and style of this paper.
 
\makeatletter
\renewcommand{\@biblabel}[1]{[#1]\hfill}
\makeatother

\end{document}